\documentclass[a4paper,abstracton]{scrartcl}
\usepackage[utf8]{inputenc}
\usepackage{amsfonts,amsthm,amssymb,amsmath}
\usepackage[inline]{enumitem}
\usepackage{color}
\usepackage{graphicx} 
\usepackage{tikz}
\usepackage{authblk}

\renewcommand{\leq}{\leqslant}

\newtheorem{thm}{Theorem}
\newtheorem{lem}[thm]{Lemma}

\newtheorem{conj}[thm]{Conjecture}

\usepackage{subcaption}
\theoremstyle{remark}

\definecolor{pink}{RGB}{219, 48, 122}

\title{Irreducible distinguishing colourings and the Axiom of Choice}           
{}                 

\author{Marcin Stawiski\\stawiski@agh.edu.pl}            
\affil{AGH University of Krakow,\\ Faculty of Applied Mathematics, \protect\\al. Mickiewicza 30, 30-059 Krakow, Poland}  
{}                     
{}

\begin{document}
\maketitle


\begin{abstract}
We say that a vertex or edge colouring of a graph is \emph{distinguishing} if the only automorphism that preserves this colouring is the identity. A (proper) distinguishing colouring is \emph{irreducible} if there is no possibility of merging two non-empty classes of colours to obtain a (proper) distinguishing colouring. We show that every graph has an irreducible (proper) distinguishing vertex colouring and that every graph without isolated edge and with at most one isolated vertex has an irreducible (proper) distinguishing edge colouring. Moreover, we show that the existence of any of these colourings for every connected graph (not isomorphic to $K_2$) is equivalent to the Axiom of Choice.
   \end{abstract}


\section{Introduction}

Let $c$ be a vertex or edge colouring of a graph $G$, and $\varphi$ be an automorphism of $G$. We say that $\varphi$ \emph{preserves} $c$ if every vertex (edge) is mapped to a vertex (edge) of the same colour. Otherwise, we say that $c$ \emph{breaks} $\varphi$. If $c$ breaks all non-identity automorphisms of $G$, then we say that $c$ is \emph{distinguishing}.  Note that a graph with an isolated edge or at least two isolated vertices cannot have a distinguishing edge colouring. Other graphs always have distinguishing edge colourings, because it is enough to colour each vertex with a unique colour. Distinguishing vertex colourings always exists.

Distinguishing vertex colourings were introduced by Babai \cite{BAB} in 1977 under the name \emph{asymmetric} colourings, which is still sometimes used. His study of this notion ultimately led to the proof of the existence of a quasi-polynomial algorithm for the Graph Isomorphism Problem \cite{babaiisomorphism}.

The \emph{von Neumann Hierarchy} (proposed by Zermelo in \cite{Zermelo}) $(V_\alpha\colon \alpha \in Ord)$ is the hierarchy defined by a transfinite induction as follows:
\begin{enumerate}
    \item $V_0=\emptyset,$
    \item $V_{\alpha+1}=\mathcal{P}(V_\alpha)$\textnormal{ for every ordinal $\alpha$,}
    \item $V_{\alpha}=\bigcup\{V_\beta\colon \beta<\alpha\}$\textnormal{ for every limit ordinal $\alpha$.}
\end{enumerate}
We define the \emph{rank} of a set $S$ as the smallest ordinal $\alpha$ such that $S\in V_{\alpha+1}$. We note that the statement "every set has a rank" is equivalent to the Axiom of Foundation in $ZF^-$ (Zermelo-Fraenkel Set Theory without the Axiom of Choice and Foundation).

Throughout the paper, we do not assume that the Axiom of Choice holds, and we work in ZF. We say that sets $X$ and $Y$ are \emph{equinumerous} if there exists a bijection from $X$ to $Y$, and we write $|X|=|Y|$ or $X\sim Y$. We define an \emph{initial  ordinal} as an ordinal that is not equinumerous with any smaller ordinal. In ZFC, every set is equinumerous with exactly one initial ordinal. In this paper, we define a \emph{cardinal number} as an equivalence class of the equinumerosity relation. The \emph{cardinality} of a set $|X|$ is the set $|X|=\{Y\colon X \sim Y,\textnormal{ and $Y$ is of the least rank}\}$. If there exists an injection from the set $X$ to $Y$, then we write $|X|\leq|Y|$ or $X \preceq Y$.

The \emph{distinguishing (chromatic)  number} of a graph $G$ is the smallest cardinality of the set of colours in a (proper) distinguishing vertex colouring of $G$, if such cardinality exists. Similarly, the \emph{distinguishing (chromatic) index} of a graph $G$ is the smallest cardinality of the set of colours in a (proper) distinguishing edge colouring of $G$, if it exists.


Galvin and Komjáth \cite{galvin} studied the notion of proper colourings in ZF. They showed that every graph has a chromatic number if and only if the Axiom of Choice holds. Stawiski \cite{StawiskiAC} studied the notions of proper and distinguishing colourings in ZF in both vertex and edge variants, and proved the following theorem.

\begin{thm}[Stawiski 2023 \cite{StawiskiAC}]\label{tw:kolorowaniaAC}
Assume the notion of cardinal number as an initial ordinal number. In ZF, the following statements are equivalent:
\begin{enumerate}[label=\textnormal{\arabic*)}]
    \item  The Axiom of Choice.\label{nwsrA:0}
    \item Every connected graph has a distinguishing number.\label{nwsrA:1j}
    \item Every connected graph has a distinguishing index unless it is isomorphic to $K_2$.\label{nwsrA:2j}
    \item Every connected graph has a chromatic index.\label{nwsrA:3j}
\end{enumerate}
\end{thm}
 The \emph{rank} of a set $X$ is the smallest ordinal $\alpha$ such that $X\in V_{\alpha+1}$. 
In ZF, we define the \emph{cardinality} $|X|$ of a set $X$ as $|X|=\{Y\colon |X|=|Y|\textnormal{ and \ensuremath{Y} is of least rank}\}$. 
We say that a set is a \emph{cardinal number} if it is the cardinality of some set.

Banerjee, Molnár, and Gopaulsingh \cite{Ban1} studied connections between K\H{o}nig’s Lemma and the existence of irreducible proper vertex colourings of locally finite connected  graphs in ZF. They proved that the following statements are equivalent in ZF by assuming the definition of a cardinal number as an equivalence class of the equinumerous relation.
\begin{thm}[Banerjee, Molnár, Gopaulsingh 2024 \cite{Ban1}] The following statements are equivalent to each other in ZF.
\begin{enumerate}[label=\textnormal{\arabic*)}]
    \item K\H{o}nig's Lemma.
    \item Every locally finite connected graph has a chromatic number.
    \item Every locally finite connected graph has a distinguishing number.
    \item Every locally finite connected graph has a distinguishing chromatic number.
    \item Every locally finite connected graph has a chromatic index.
    \item Every locally finite connected graph not isomorphic to $K_2$ has a distinguishing  index.
    \item Every locally finite connected graph not isomorphic to $K_2$ has a distinguishing chromatic index.
\end{enumerate}
\end{thm}

Furthermore, the question mentioned below is still open to the best of our knowledge.

\begin{conj} The following statements are equivalent the Axiom of Choice in ZF assuming the definition of a cardinal number as  an equivalence class of the equinumerous relation.
\begin{enumerate}[label=\textnormal{\arabic*)}]
    \item Every infinite connected graph has a chromatic index.
    \item Every infinite connected graph has a distinguishing chromatic number.
    \item Every infinite connected graph has a distinguishing chromatic index.
    \item Every infinite connected graph has a distinguishing number.
    \item Every infinite connected graph has a distinguishing chromatic index.
\end{enumerate}
\end{conj}

Banerjee, Molnár, and Gopaulsingh \cite{Ban2} asked if the conditions \ref{nwsrA:1j}--\ref{nwsrA:3j} of Theorem \ref{tw:kolorowaniaAC} imply the Axiom of Choice if we instead use the definition of a cardinal number due to Scott i.e. based on the concept of a rank.

Let $c\colon X\rightarrow Y$ be a vertex (edge) colouring of a graph $G$ with some property $\Phi$. Let $a,b \in Y$ be two colours and let $\sigma$ be a function from $Y$ to itself such that $\sigma(b)=a$ and $\sigma(x)=x$ for $x\neq b$. A \emph{reduction} of two (not necessarily different) colours $a,b$ is a composition of $\sigma \circ c$. A process of obtaining a reduction from a given colouring is called \emph{merging}. In other words, a colouring $c'$ is a reduction of $c$ if it is equal to $c$, or there exist colours $c_1,c_2$ used in $c$ such that each vertex (edge) of colour $c_1$ in $c$ has the colour $c_2$  in $c'$ and the remaining colours in $c'$ are the same as in $c$. We say that $c$ is \emph{irreducible}  (with respect to the property $\Phi$) if there is no non-trivial reduction of $c$ with the property $\Phi$. Connections between the Axiom of Choice and irreducible proper vertex colourings were studied by Galvin and Komjáth \cite{galvin}. They proved that the Axiom of Choice is equivalent to the statement "Every connected graph has an irreducible proper vertex colouring" in ZF.


The main theorem of the paper is as follows:
\begin{thm}\label{thm:main1}
Let $G$ be a graph, and $c$ be a (proper) distinguishing vertex (edge) colouring of $G$ with a well-orderable set of colours. Then there exists a reduction $c'$ of $c$ which is an irreducible (proper) distinguishing vertex (edge) colouring.
\end{thm}

Note that this has not been previously known to hold for distinguishing colourings, even if the Axiom of Choice is assumed.

The second main theorem gives some statements, which are equivalent to the Axiom of Choice in ZF.

\begin{thm}\label{thm:main2}
In ZF the following statements are equivalent:
\begin{enumerate}[label=\textnormal{\arabic*)}]
\item The Axiom of Choice.
    \item Each connected graph has an irreducible distinguishing vertex colouring.
     \item Each connected graph has an irreducible distinguishing edge colouring.
     \item Each connected graph has an irreducible proper distinguishing vertex colouring.
     \item Each connected graph has an irreducible proper distinguishing edge colouring.
\end{enumerate}
\end{thm}

More about the equivalent forms of the Axiom of Choice may be found in the extensive monograph by Rubin and Rubin \cite{Rubin}. For notions of graph theory, which are not presented here, see the monograph by Diestel \cite{Diestel}. For notions of set theory, which are not presented here, see the monograph by Jech \cite{Jech}.

\section{Well orderable-sets of colours and irreducible distinguishing colourings}\label{chapter:ac}

In this section, we prove Theorem \ref{thm:main1}.

\begin{proof} We shall handle all possible versions of the statement in the theorem simultaneously. We say that a colouring $c$ is \emph{suitable} if it is a (proper) distinguishing vertex (edge) colouring depending on the considered statement.
   Let $G$ be a graph, and $c$ be a suitable colouring with a well-orderable set $\kappa$ of colours. Without loss of generality, we can assume that $\kappa$ is a cardinal number defined as an initial ordinal. We construct $c'$ from $c$ by induction on $\alpha \leq \kappa$.  In step $\alpha$ we assume that we have colourings $(c_\gamma\colon \gamma\leq \alpha)$ with the following properties:
    \begin{enumerate}
        \item For every $\gamma\leq \alpha$ the colouring $c_\gamma$ is distinguishing.
        \item For every $\gamma\leq \alpha$ the colouring $c_\gamma$ is proper if the statement requires a proper colouring.
        \item For every $\delta<\gamma\leq \alpha$ the colouring $c_\gamma$ is a reduction of $c_\delta$. 
        \item For every $\gamma<\alpha$ the colouring $c_{\gamma+1}$ is equal to $c_\gamma$ or $c_{\gamma+1}$ can be obtained from $c_\gamma$ by merging some classes of colours greater than $\gamma$ with $\gamma$.
    \end{enumerate}
    
    We put $c_0=c$. All the conditions above are satisfied for $c_0$. Let $S_\alpha(0)=\{c_0\}$, and $c_\alpha^0=c_\alpha$. In step $\alpha$, we consider all the colours $\beta$ for $\alpha<\beta \leq \kappa$ by induction on $\beta$. We shall modify the colouring $c_\alpha$ at some stages of this induction, and we shall define colourings $(c_\alpha^\beta\colon \alpha<\beta\leq \kappa)$, which will correspond to the currently considered colouring obtained from $c_\alpha$ by some reduction operations.

    Assume first that $\beta$ is a successor ordinal. In the construction, we define the set $S_\alpha(\beta)$, which is useful for us when we consider limit ordinals $\beta'$ in a different step $\beta'$ of induction and its exact role shall be explained later.  We check, whether the colouring obtained by merging $\beta$ with $\alpha$ is still a distinguishing (proper) colouring. If not, we put $S_\alpha(\beta)=S_\alpha(\beta-1)$. Otherwise, we construct a new colouring $c_\alpha^{\beta+1}$ by changing all the occurrences of $\beta$ in the current colouring $c_\alpha^\beta$ with $\alpha$. We put $S_\alpha(\beta)=S_\alpha(\beta-1)\cup \{c_\alpha^{\beta+1}\}$. 
    
    Assume now that $\beta$ is a limit ordinal. Let $\delta<\beta$. 
    Assume that $S_\alpha(\beta)$ is a non-empty set of (proper) distinguishing colourings such that for every pair of elements $\{x,y \}$ of $S_\alpha(\beta)$ either $x$ is a reduction of $y$ or $y$ is a reduction of $x$. This is indeed satisfied for $S_\alpha(0)$ and shall be satisfied after every step $\beta$ of induction. The role of the set $S_\alpha(\beta)$ is to construct tails of $S_\alpha(\beta)$. Some of the reduction operations in steps for limit ordinals $\beta$ may result in a colouring not being distinguishing. The role of tails and non-vanishing colourings, which we define in a moment, is to restrict ourselves to reductions which are still distinguishing. The \emph{tail} $T(\beta,\delta)$ of $S_\alpha(\beta)$ is a set of these colourings $c''$ in which for every colour $k$ a vertex (edge) $v$ has colour $k$  if and only if $v$ has the colour $k$ in all the colourings in $S_\alpha(\gamma)$ for $\beta>\gamma>\delta$. In other words, $c''$ is a limit colouring. We say that colour $\gamma<\kappa$ is \emph{non-vanishing} if its class is non-empty in all the colourings in $T(\beta,\delta)$ for some $\delta<\beta$. From the construction, it shall follow that each vertex (edge) has exactly one non-vanishing colour, because the class of available colours is well-ordered, and the colour of a vertex (edge) can decrease only a finite number of times. Consider the colouring $c^*$ in which each vertex (edge) has its non-vanishing colour. Notice that if the considered statement refers to proper colourings, then $c^*$ is proper. However, it may happen that $c^*$ is not distinguishing.
    
    Let $\Gamma$ be the set of non-trivial automorphisms of $G$ that are preserved by $c^*$. Assume that $\Gamma$ is non-empty. Consider an automorphism $\varphi \in \Gamma$. It follows that before the limit step at some point there exists a vertex (edge) which is mapped to a vertex (edge) of a different colour, and either one of them is $\alpha$ and the other one is merged with $\alpha$, or both are merged in $\alpha$.
    Therefore, there exists a unique maximal set $A_\varphi\subseteq \kappa \setminus (\alpha+1)$ such that if $a \in A_\varphi$, then there exists a vertex (edge) of $G$ that is mapped by $\varphi$ to a vertex (edge) of a colour $a^*\neq a$ in $A_\varphi$ in some colouring $k\in T(\alpha,\delta)$ for some $\delta<\beta$.

    We shall now describe the procedure of reversing the merging. We shall briefly explain the idea behind it. If we just recolour some vertices (or edges) it may happen that the standard limit colouring preserves some automorphism of a graph even if no colouring used in the definition of this limit colouring preserves this automorphism. The idea of reversing the merging of colouring $c$ is to pick a minimal set $K$ of colours for which we do not perform the merging operation. More formally, we obtain the colouring $c^*$ from a colouring $c$ by setting $c^*(x)=x$ for each $x$ such that $c(x) \in K$ and setting $c^*(x)=c(x)$ for all the remaining vertices (edges) $x$.

If we reverse the merging operation on some colour in $A_\varphi$, then the resulting colouring would not preserve $\varphi$. We would like to obtain a minimal subset $B'$ of $B=\bigcup  \{A_\varphi\colon \varphi \in \Gamma\}$ such that if we reverse the merging operation between the elements in $B'$ and $\alpha$, then we break all the automorphisms in $\Gamma$. We cannot just take one element of every set $A_\varphi$ as this alone could break more than one automorphism, and some of the chosen elements would be redundant.
    
    We now proceed with the construction of $B'$ by induction on $a<\beta$. We start with an empty set $B_0$. In step $a$ if $a\notin B$, then we do nothing. If $a\in B$, then we consider the colouring $k_a$ obtained from $c^*$ by reversing the operation of merging $b$ with $\alpha$ for every $b\in B_\alpha$ i.e. the colouring $k_a$ assigns colour $\alpha$ to vertices which which are not in $B_\alpha$ and $k_a=c^*$ otherwise. If $a$ is in $B$ and there exists an automorphism $\varphi\in \Gamma$ that is still preserved by the current colouring, then we put $B_a=\bigcup\{B_b\colon b<\alpha\}\cup \{a\}$. Otherwise, we put $B_a=\bigcup\{B_b\colon b<\alpha\}$. We proceed in a similar way as in the step for a successor ordinal: We check whether the colouring obtained by merging $\beta$ and $\alpha$ is still a distinguishing (proper) colouring. If not, then we do nothing. Otherwise, we change all the occurrences of $\beta$ with $\alpha$. We obtained a colouring $c_\alpha^{\beta}$. In contrast to the step for a successor ordinal, we put $S_\alpha(\beta)=\{c_\alpha^{\beta}\}$.
     
     We claim that the obtained colouring is a distinguishing colouring. Suppose that this is not the case. Hence, there exists an automorphism $\varphi$ which is preserved by $c_\alpha^\beta$. The colouring $c^*$ breaks
$\varphi$, and the subgroup $\Gamma'$ of automorphisms preserved by $c_\alpha^\beta$ is a subgroup of $\Gamma$. Hence, $\varphi \in \Gamma$. There exists a colour $a\in B$ such that $\varphi$ is not preserved by the colouring obtained from $c^*$ by fixing $B_a$. Similarly as above, the group of automorphisms preserved by this colouring is a subgroup of $\Gamma'$. Hence, $\varphi\notin \Gamma'$. We proved that $c_\alpha^\beta$ is distinguishing.

     After induction on all $\beta<\kappa$, we obtain a limit colouring $c_{\alpha}=c_\alpha^\kappa$. From the construction and the paragraph above, it follows that the set $(c_\gamma \colon \gamma \leq \alpha)$ satisfies the desired properties.


     After induction on all $\alpha$, we obtain a limit colouring $c'=c_\kappa$. We have already shown that it is distinguishing and proper. It remains to show that it is irreducible. Suppose that this is not the case and that there exist colours $\alpha, \beta$ for some $\alpha<\beta<\kappa$ such that $\beta$ can be merged with $\alpha$ to obtain a (proper) distinguishing colouring. However, that means that $\beta$ would be merged with $\alpha$ in $c_\alpha^\kappa$. This is a contradiction.
\end{proof}

\section{Irreducible colourings and the Axiom of Choice}

In this section, we prove Theorem \ref{thm:main2} using Lemma \ref{lem:lemat2xstar}. We first need some definitions to state this lemma.
Let $X,Y,\{x' \}$ and $\{y' \}$ be pairwise disjoint non-empty sets. Define \emph{double star} $DS(X,Y)$ as the graph with the vertex set $V(DS(X,Y))=X \cup Y \cup \{x',y'\}$ such that $e$ is an edge in $DS(X,Y)$ if and only if $e=x'y'$, $e=x'x$, or $e=y'y$ for all $x \in X$ and all $y\in Y$. A \emph{double clique} $DC(X,Y)$ is the graph obtained from $DS(X,Y)$ by adding edges between each pair of different vertices in $X$ and between each pair of different vertices in $Y$.

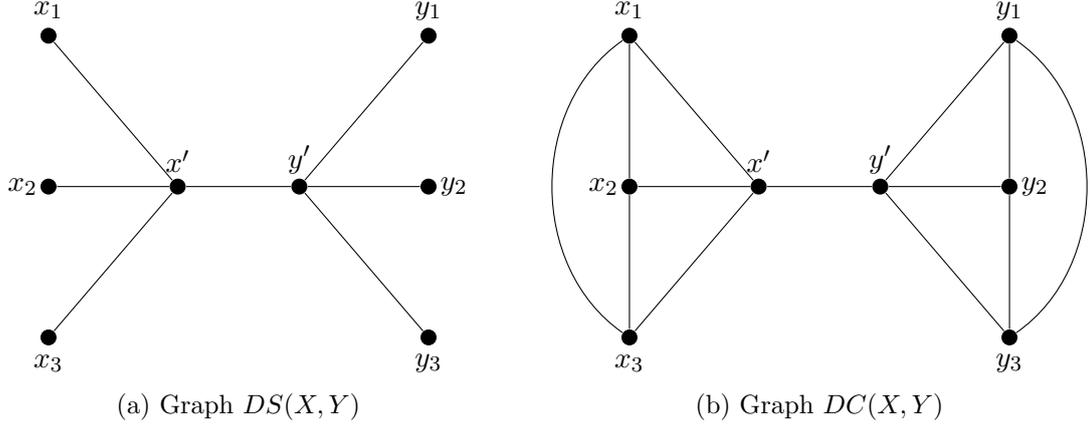
\begin{figure}[t] \centering \begin{subfigure}{0.48\textwidth} \centering \resizebox{!}{5cm}{ \begin{tikzpicture}[ every node/.style={ circle, fill=black, minimum size=6pt, inner sep=0pt } ] \path[use as bounding box] (-3,-2.5) rectangle (3,2.5); \node[label=above:$x_1$] (x1) at (-2.5, 2) {}; \node[label=left:$x_2$] (x2) at (-2.5, 0) {}; \node[label=below:$x_3$] (x3) at (-2.5,-2) {}; \node[label=above:$x'$] (xp) at (-0.8,0) {}; \node[label=above:$y'$] (yp) at (0.8,0) {}; \node[label=above:$y_1$] (y1) at (2.5, 2) {}; \node[label=right:$y_2$] (y2) at (2.5, 0) {}; \node[label=below:$y_3$] (y3) at (2.5,-2) {}; \draw (x1) -- (xp); \draw (x2) -- (xp); \draw (x3) -- (xp); \draw (xp) -- (yp); \draw (yp) -- (y1); \draw (yp) -- (y2); \draw (yp) -- (y3); \end{tikzpicture} } \caption{Graph $DS(X,Y)$} \end{subfigure} \hfill \begin{subfigure}{0.48\textwidth} \centering \resizebox{!}{5cm}{ \begin{tikzpicture}[ every node/.style={ circle, fill=black, minimum size=6pt, inner sep=0pt } ] \path[use as bounding box] (-3,-2.5) rectangle (3,2.5); \node[label=above:$x_1$] (x1) at (-2.5, 2) {}; \node[label=left:$x_2$] (x2) at (-2.5, 0) {}; \node[label=below:$x_3$] (x3) at (-2.5,-2) {}; \node[label=above:$x'$] (xp) at (-0.8,0) {}; \node[label=above:$y'$] (yp) at (0.8,0) {}; \node[label=above:$y_1$] (y1) at (2.5, 2) {}; \node[label=right:$y_2$] (y2) at (2.5, 0) {}; \node[label=below:$y_3$] (y3) at (2.5,-2) {}; \draw (x1) -- (xp); \draw (x2) -- (xp); \draw (x3) -- (xp); \draw (xp) -- (yp); \draw (yp) -- (y1); \draw (yp) -- (y2); \draw (yp) -- (y3); \draw (x1) -- (x2); \draw (x2) -- (x3); \draw[bend right=55] (x1) to (x3); \draw (y1) -- (y2); \draw (y2) -- (y3); \draw[bend left=55] (y1) to (y3); \end{tikzpicture} } \caption{Graph $DC(X,Y)$} \end{subfigure} \end{figure}

\begin{lem}\label{lem:lemat2xstar}
Let $X$ and $Y$ be disjoint non-empty sets. Then the following conditions are equivalent.
\begin{enumerate}[labelindent=\parindent,leftmargin=*, label=\textnormal{\alph*)}]
    \item $DS(X,Y)$ has an irreducible distinguishing vertex colouring.  \label{nwsr:lemat2xstarD}
    \item  $DS(X,Y)$ has an irreducible distinguishing edge colouring.   \label{nwsr:lemat2xstarDe}
    \item  $DC(X,Y)$ has an irreducible proper distinguishing vertex colouring. \label{nwsr:lemat2xstarCh}
    \item  $DS(X,Y)$ has an irreducible proper distinguishing edge colouring.  \label{nwsr:lemat2xstarChe}
    \item $X\preceq Y$ or $Y \preceq X$.  \label{nwsr:lemat2xstar5}
\end{enumerate}
\end{lem}
\begin{proof}
  We denote the image of a set $X$ (with respect to the function $c$) by $c[X]$.

$\mathbf{(a\Rightarrow e):}$
Assume that  there exists an irreducible distinguishing vertex colouring $c$ of $DS(X,Y)$.
In particular, it means that every vertex in $X$ has a different colour  and every vertex in $Y$ has a different colour. Hence, $|c[X]|=|X|$ and $|c[Y]|=|Y|$. Suppose that condition \ref{nwsr:lemat2xstar5} does not hold. Hence, there exists a vertex $x \in X$ such that 
$c(x) \notin c[Y]$ and a vertex $y \in Y$ such that $c(y) \notin c[X]$.  Let $c'$ be the colouring obtained from $c$ by recolouring every vertex of colour $c(x)$ with the colour $c(y)$. Now $x'$ and $y'$ cannot be swapped because they either have different colours or they have different set of colours on their neighbours.
The colouring $c'$ is a non-trivial distinguishing reduction of $c$, because $x'$ and $y'$ cannot be swapped, and $X$ and $Y$ are fixed pointwise. This contradicts the assumption of the irreducibility of $c$.

$\mathbf{(a\Rightarrow b):}$
Assume that  there exists an irreducible distinguishing vertex colouring $c$ of $DS(X,Y)$. We colour each edge $xx'$ with colour $c(x)$ and each edge $yy'$ with colour $c(y)$. We colour $x'y'$ with one of the used colours. The obtained colouring is an irreducible distinguishing edge colouring of $DC(X,Y)$. 

$\mathbf{(b\Rightarrow a):}$
Assume that  there exists an irreducible distinguishing edge colouring $c$ of $DS(X,Y)$. We colour each vertex $x'$ with colour $c(xx')$ and each vertex $y'$ with colour $c(yy')$. We colour $x$ and $y$ with different colours used before if $|X|>1$ or $|Y|>1|$, or the same if $|X|=|Y|=1$.

$\mathbf{(e\Rightarrow a):}$
Assume now that \ref{nwsr:lemat2xstar5} holds. Without loss of generality, we  can assume that $X\preceq Y$ and that $Y$ has at least two elements. Let $f$ be an injection from $X$ to $Y$. Let $c(x)=x$, and $c(f(x))=c(x)$ for $x \in X$. We colour the uncoloured vertices of $Y$ with different unused colours. If $f$ is a bijection, then there exists an automorphism $\varphi$ of $DS(X,Y)$ that maps $X$ into $Y$. In that case, we colour $x'$ and $y'$ in $c$ with arbitrary but different colours from the set $c[Y]$. Note that each such automorphism $\varphi$ is not preserved by $c$. If $f$ is not a bijection, then we colour $x'$ and $y'$ using the same colour from set $c[Y]$. The defined colouring $c$ is distinguishing. We shall show that it is also irreducible. The set $Y$ is fixed pointwise in $c$. In particular, each element of $Y$ has to have a unique colour. Therefore, we cannot merge any two colours of $c[Y]$. Since every colour in $c$ lies in $c[Y]$, colouring $c$ is irreducible.

$\mathbf{(a\Rightarrow c):}$
Recall that $DS(X,Y)$ and $DC(X,Y)$ have the same set of vertices.
 Let $c$ be an irreducible distinguishing vertex colouring of $DS(X,Y)$. If $c[X]=c[Y]$, then we recolour $x'$ and $y'$ with new, distinct colours to obtain an irreducible proper distinguishing vertex colouring $c'$ of $DC(X,Y)$. 
 Otherwise, there exists a colour pink, which is in $c[X]\setminus c[Y]$ or in $c[Y]\setminus c[X]$. Without
loss of generality, assume that pink lies in $c[X] \setminus c[Y]$.
  We colour $y'$ with pink, and we colour
$x'$ either with a colour from $c[Y]\setminus c[X]$  if this set is non-empty, or with a new colour if $c[Y]\setminus c[x]$ is empty. The obtained colouring $c'$ is an irreducible proper distinguishing vertex colouring of the graph $DS(X,Y)$.

 $\mathbf{(c\Rightarrow a):}$
Let $c$ be an irreducible proper vertex colouring of $DC(X,Y)$. As $c(x')\neq c(y')$, colouring $c$ is a distinguishing colouring of $DS(X,Y)$. If $c(X)=c(Y)$, then we colour $x'$ with an arbitrary colour from $c(X)$, and $y'$ with a new colour. The colouring $c$ is an irreducible distinguishing colouring. Assume now that $c[X]\neq c[Y]$. Without loss of generality, we can assume that there exists a colour pink in $c(X)\setminus c(Y)$. If $c[X] = c[Y ]$, then
we recolour $x'$ and $y'$ as follows: let $c(x')$ be an arbitrary colour from $c[X]$ and $c(y')$ be a
new colour. Then, after recolouring $x'$ and $y'$, $c$ is an irreducible distinguishing colouring of $DC(X,Y)$.

 $\mathbf{(b\Rightarrow d):}$ The proof of the equivalence between \ref{nwsr:lemat2xstarDe} and \ref{nwsr:lemat2xstarChe} is essentially the same as that of the equivalence between a) and b). We obtain the desired colourings from each other by recolouring the edge $x'y'$.
 
  $\mathbf{(c\Rightarrow d):}$ Let $c$ be a proper irreducible distinguishing vertex colouring of $DC(X,Y)$. We construct an edge colouring $c'$ in the following way. We assign colour $c(x)$ to the edge $x'x$  for each $x\in X$, and we assign the colour $c(y)$ to the edge $y'y$  for each $y\in Y$. Furthermore, we assign a new colour to the edge $x'y'$. Either this gives us an irreducible proper distinguishing edge colouring of $DC(X,Y)$, or we can
obtain such by recolouring one of the edges of the form $xx'$ with a new colour if $X \sim Y$, because if $X \sim Y$ and $c(X)=c(Y)$, then we have to recolour one of the edges of such form to a new colour as it is the necessary and sufficient way to distinguish $X$ from $Y$.

$\mathbf{(d\Rightarrow c):}$ Similarly, as in the paragraph mentioned in the proof of $\mathbf{(c\Rightarrow d)}$, if $c'$ is an irreducible proper distinguishing edge colouring of $DS(X,Y)$, then we assign to each vertex $y\in Y$,  colour $c(y)=c'(yy')$, and to each vertex $x\in X$,  colour $c(x)=c(xx')$.
We colour $x'$ in $c$ with an arbitrary colour from the set $c[Y] \setminus c[X]$ if $c[Y] \setminus c[X]$ is non-empty or with a new colour if this set is empty, and we colour $y'$ in $c$ with an arbitrary colour
from the set $c[X]\setminus c[Y]$ if $c[X]\setminus c[Y]$ is non-empty or a new colour different from $c(x')$ if
$c[X]\setminus c[Y]$ is empty. The colouring $c$ is an irreducible proper distinguishing vertex colouring of $DC(X,Y)$. This concludes the proof of the equivalence between conditions \ref{nwsr:lemat2xstarCh} and \ref{nwsr:lemat2xstarChe}, and thus the proof of the lemma.\end{proof}

Now we prove Theorem \ref{thm:main2}.
\begin{proof}
Assume that there exists an irreducible (proper) distinguishing colouring of a graph $DS(X,Y)$ ($DC(X,Y)$ respectively) for every disjoint sets  $X$ and $Y$. From Lemma \ref{lem:lemat2xstar}, we see that either $X \preceq Y$ or $Y \preceq X$ holds. Hence, the Law of Trichotomy holds, because the choice of $X$ and $Y$ was arbitrary, and so does the Axiom of Choice. This observation, together with Theorem \ref{thm:main1}, gives us Theorem \ref{thm:main2}.\end{proof}

\bibliographystyle{abbrv}
\bibliography{lit.bib}












\end{document}